\documentclass[a4paper]{article}

\usepackage[english]{babel}
\usepackage[T1]{fontenc}
\usepackage{amssymb}
\usepackage{amsthm}
\usepackage{makeidx}
\usepackage{color}
\usepackage{fullpage}
\usepackage{tikz}
\usepackage{mathtools}
\usepackage{multirow}
\usepackage[all]{xy}
\usepackage{comment}
\usepackage{authblk}
\usepackage[backend=biber,style=numeric,natbib=true]{biblatex}

\usepackage[shortlabels]{enumitem}
\usepackage{csquotes}

\usepackage[a4paper,top=3cm,bottom=2cm,left=3cm,right=3cm,marginparwidth=1.75cm]{geometry}

\usepackage{amsmath}
\usepackage{graphicx}
\usepackage[colorlinks=true, allcolors=blue]{hyperref}
\usepackage[all]{xy}

\providecommand{\MSC}[1]{\textbf{\textit{MSC---}} #1}

\newtheorem{defi}{Definition}[section]

\newtheorem{thm}[defi]{Theorem}
\newtheorem{coro}[defi]{Corollary}
\newtheorem{lem}[defi]{Lemma}

\newtheorem{pro}[defi]{Proposition}

\newtheorem{prob}[defi]{Problem}

\title{Hopf-Galois module structure of \\ monogenic orders in cubic number fields}
\author[1,2]{Daniel Gil-Muñoz}
\affil[1]{Department of Algebra, Faculty of Mathematics and Physics, Charles University,
Sokolovsk\'{a} 83, 186 00 Praha 8, Czech Republic\bigskip}
\affil[2]{Dipartimento di Matematica, Università di Pisa, Largo B. Pontecorvo, 5, 56127 Pisa, Italy\bigskip}

\date{}

\begin{document}
\maketitle

\begin{abstract}
For a cubic number field $L$, we consider the $\mathbb{Z}$-order in $L$ of the form $\mathbb{Z}[\alpha]$, where $\alpha$ is a root of a polynomial of the form $x^3-ax+b$ and $a,b\in\mathbb{Z}$ are integers such that $v_p(a)\leq 2$ or $v_p(b)\leq 3$ for all prime numbers $p$. We characterize the freeness of $\mathbb{Z}[\alpha]$ as a module over its associated order in the unique Hopf-Galois structure $H$ on $L$ in terms of the solvability of at least one between two generalized Pell equations in terms of $a$ and $b$. We determine when the equality $\mathcal{O}_L=\mathbb{Z}[\alpha]$ is satisfied in terms of congruence conditions for $a$ and $b$. For such cases, we specialize our result so as to obtain criteria for the freeness of $\mathcal{O}_L$ as a module over its associated order in $H$.
\end{abstract}

\MSC{11A05,11A55,11R04,11R16,11R32,11R33,11Y40,12F10,16T05}

\section{Introduction}

Galois module theory concerns the structure of the ring of integers of a number or a $p$-adic field as a module over some object depending on a Galois action. The starting point of Galois module theory is the normal basis theorem, which asserts that a Galois field extension admits a basis formed by the conjugates of a single element, which is a so called normal basis. For a Galois number field $L$ with group $G$, this can be rephrased by saying that $L$ is $\mathbb{Q}[G]$-free of rank one. 

A more subtle question asks whether $L$ admits some normal basis which simultaneously is integral, i.e. a normal integral basis. This condition is equivalent to saying that the ring of integers $\mathcal{O}_L$ of $L$ is $\mathbb{Z}[G]$-free of rank one. The first major result in this direction is the Hilbert-Speiser theorem, which ensures that $L$ admits some normal integral basis if and only if $L/\mathbb{Q}$ is tamely ramified. This also implies the non-existence of a normal integral basis on wildly ramified extensions. In that case, the ground ring for the module structure of the ring of integers is replaced by another one. Namely, instead of $\mathbb{Z}[G]$ we use $$\mathfrak{A}_{L/\mathbb{Q}}=\{\lambda\in\mathbb{Q}[G]\,|\,\lambda\cdot x\in\mathcal{O}_L\hbox{ for all }x\in\mathcal{O}_L\},$$ which is called the associated order of $\mathcal{O}_L$ in $\mathbb{Q}[G]$. It trivially endows $\mathcal{O}_L$ with module structure. 

It is not difficult to check that the associated order $\mathfrak{A}_{L/\mathbb{Q}}$ of $\mathcal{O}_L$ is the only $\mathbb{Z}$-order over which $\mathcal{O}_L$ is possibly free. However, $\mathcal{O}_L$ is not $\mathfrak{A}_{L/\mathbb{Q}}$-free in general. The problem then turns to find criteria to characterize the freeness of $\mathcal{O}_L$ as a module over the associated order. Leopoldt \cite{leopoldt} proved that if the number field $L$ is abelian, then $\mathcal{O}_L$ is $\mathfrak{A}_{L/\mathbb{Q}}$-free of rank one. Some other works \cite{berge1972,ferri2024,martinet1972} have found characterizations of freeness for classes of non-abelian number fields, even though a general answer is currently out of reach.

The picture of Galois module theory can be broadened by means of Hopf-Galois theory, which is a generalization of Galois theory with the use of Hopf algebras. The basic notion in this approach is the one of Hopf-Galois structure on a finite field extension $L/K$, a pair formed by a $K$-Hopf algebra $H$ and a $K$-linear action $\cdot\colon H\otimes_KL\longrightarrow L$ such that $L$ is $H$-module algebra with respect to $\cdot$ and the canonical map $L\otimes_KH\longrightarrow\mathrm{End}_K(L)$ is a $K$-linear isomorphism. A Hopf-Galois extension is nothing but a field extension admitting some Hopf-Galois structure. With this definition, every Galois extension is Hopf-Galois; indeed, calling $G\coloneqq\mathrm{Gal}(L/K)$, $K[G]$ together with its action by endomorphisms on $L$ is a Hopf-Galois structure on $L/K$. The beginning of Hopf-Galois theory traces back to the book \cite{chasesweedler} by Chase and Sweedler.

Now, let $L$ be a number field such that $L/\mathbb{Q}$ is Hopf-Galois and let $H$ be a Hopf-Galois structure on $L/\mathbb{Q}$. In short, we will say that $L$ is an $H$-Galois number field. The associated order of $\mathcal{O}_L$ in $H$ is defined as the set $$\mathfrak{A}_H=\{\lambda\in H\,|\,\lambda\cdot x\in\mathcal{O}_L\hbox{ for all }x\in\mathcal{O}_L\}.$$ This is a natural generalization of the notion of associated order from the Galois case. In fact, it also endows $\mathcal{O}_L$ with $\mathfrak{A}_H$-module structure, and it makes sense to ask whether $\mathcal{O}_L$ is $\mathfrak{A}_H$-free of rank one.

Thus, we have the more general question: given a Hopf-Galois structure $H$ on an extension $L/K$, is $\mathcal{O}_L$ free as an $\mathfrak{A}_H$-module? The first time that this problem was addressed for non classical Hopf-Galois structures on extensions number fields was due to Truman \cite{truman2012}, who found criteria for the $\mathfrak{A}_H$-freeness of $\mathcal{O}_L$ in the case that $L$ is a tamely ramified biquadratic number field. His strategy consist roughly in studying the local freeness of the ring of integers (that is, the freeness of the completion of the ring of integers at each rational prime over the completion of the associated order in a Hopf-Galois structure) and using this to extract criteria for the freeness by means of the theory of id\`eles. 

In \cite{gilrioinduced}, Rio and the author introduced a completely different strategy to study the $\mathfrak{A}_H$-module structure of $\mathcal{O}_L$. Let $n$ be the degree of $L$. The idea is that from the knowledge of the action of a $\mathbb{Q}$-basis $W$ of $H$ on an integral basis $B$ of $L$, one can build a $n^2\times n$ matrix $M(H_W,L_B)$ that can be transformed by means of integral linear transformations to a square matrix. If $M(H_W,L_B)$ has integer coefficients, this amounts to finding its Hermite normal form. The determinant of such a square matrix is the module index $[\mathfrak{A}_H:\mathfrak{H}]$, where $\mathfrak{H}$ is the $\mathbb{Z}$-lattice generated by $W$. Now, by analyzing the action of $W$ on a given element $\beta\in\mathcal{O}_L$, one can also obtain the module index $[\mathcal{O}_L:\mathfrak{A}_H\cdot\beta]$, and then obtain criteria on whether $\beta$ is an $\mathfrak{A}_H$-free generator of $\mathcal{O}_L$.

This method was successfully applied in \cite{gilrioquartic} so as to determine a characterization for the freeness of $\mathcal{O}_L$ in the case of arbitrary quartic Galois number fields, extending the results by Truman. It also plays a role in \cite{gilkummer}, where the author found sufficient conditions for subfamilies of radical number fields of the form $L=\mathbb{Q}(\sqrt[n]{a})$.

It is actually possible to consider a more general version of the aforementioned problem, working with a $\mathbb{Z}$-order $\mathcal{O}$ of an $H$-Galois number field $L$ rather than just the ring of integers. There is also a notion of associated order $\mathfrak{A}_H(\mathcal{O})$ in $H$ satisfying analogous properties as the usual associated order, and then one may wonder whether $\mathcal{O}$ is $\mathfrak{A}_H(\mathcal{O})$-free. We shall adapt the method by Rio and the author to this more general case, giving rise to a strategy to obtain criteria for the freeness of a $\mathbb{Z}$-order. In this case, the input is the action of a $\mathbb{Q}$-basis $W$ of $H$ on a $\mathbb{Z}$-basis $B$ of $\mathcal{O}$, and the Hermite normal form of the arising matrix $M(H_W,L_B)$ is determined.

In this paper, we consider cubic number fields $L$, which are known to be Hopf-Galois and admit a unique Hopf-Galois structure $H$. By using the generalizarion of the already mentioned method, we shall determine criteria for the $\mathfrak{A}_H(\mathbb{Z}[\alpha])$-freeness of the monogenic $\mathbb{Z}$-orders $\mathbb{Z}[\alpha]$ in $\mathcal{O}_L$, where $\alpha$ is a primitive element for $L/\mathbb{Q}$ with minimal polynomial of the form $$f(x)=x^3-ax+b,\quad a,b\in\mathbb{Z},\quad v_p(a)\leq2\hbox{ or }v_p(b)\leq3\quad\hbox{for every prime }p.$$ Every cubic number field $L$ admits an element $\alpha$ with this property. Our main result is the following:

\begin{thm}\label{thm:maintheorem} Let $L=\mathbb{Q}(\alpha)$, where $\alpha$ is a root of a polynomial as above. Write $\Delta=4a^3-27b^2$ for the discriminant of $f$ and call $g=\gcd(a,b)$. Let $H$ be the only Hopf-Galois structure on $L/\mathbb{Q}$. The $\mathbb{Z}$-order $\mathbb{Z}[\alpha]$ of $L$ is $\mathfrak{A}_H(\mathbb{Z}[\alpha])$-free if and only if there exist integers $x,y\in\mathbb{Z}$ such that:
\begin{enumerate}
    \item $x^2+3\Delta y^2=\pm12ag$, $6a\mid 9by+x$ or $9by-x$ and $3\nmid y$, if $3\nmid a$.
    \item $x^2+3\Delta y^2=\pm36ag$ and $6a\mid 9by+x$ or $9by-x$, if $3\mid a$ and $v_3(a)\leq v_3(b)$.
    \item $x^2+3\Delta y^2=\pm108ag$ and $6a\mid 9by+x$ or $9by-x$, if $v_3(a)>v_3(b)$.
\end{enumerate}
\end{thm}

Given an element $\alpha$ as above, the equality $\mathcal{O}_L=\mathbb{Z}[\alpha]$ can be characterized in terms of congruence conditions on $a$ and $b$. This can be deduced from the classification of an integral basis for a cubic number field due to Alaca \cite{alaca}. When such conditions are satisfied, Theorem \ref{thm:maintheorem} provides criteria for the ring of integers $\mathcal{O}_L$ of $L$ as a module over its associated order $\mathfrak{A}_H$. The criteria obtained for this case is stated in Corollary \ref{coro:freenesscubic}. Recall that $L/\mathbb{Q}$ is Galois if and only if the discriminant $\Delta=4a^3-27b^2$ is a square, which is not typically satisfied even if $\mathcal{O}_L=\mathbb{Z}[\alpha]$. Thus, this is the first time that criteria for the freeness on non-Galois cubic number fields is provided.

This paper is organized as follows. In Section \ref{sec:redmethod}, we prove the validity of the generalization of the method by Rio and the author to the case of $\mathbb{Z}$-orders in a number field $L$. Since we are considering cubic fields whose primitive element is a root of a trinomial, we will need to compute Hermite normal forms of matrices with two integer parameters on the same column, and how this affects the entries in the next column. In Section \ref{sec:euclidalg}, we find an answer to this problem by applying Euclid algorithm on the entries in the same column, as well as the continued fraction expansion of their quotient. 

We shall introduce the setting we use for cubic number fields in Section \ref{sec:cubicfields}, where we also provide a $\mathbb{Q}$-basis $W$ for its unique Hopf-Galois structure $H$. In Section \ref{sec:assocorders}, we carry out the application of the method from Section \ref{sec:redmethod} for the $\mathbb{Z}$-orders in $\mathbb{Z}[\alpha]$ with $W$ and the $\mathbb{Z}$-basis $B=\{1,\alpha,\alpha^2\}$ of $\mathbb{Z}[\alpha]$. We shall calculate the Hermite normal form of the matrix $M(H_W,L_B)$, and we find a $\mathbb{Z}$-basis for $\mathfrak{A}_H(\mathbb{Z}[\alpha])$ at each of the cases in Theorem \ref{thm:maintheorem}. In Section \ref{sec:freeness}, we address the problem of the freeness of $\mathbb{Z}[\alpha]$ as an $\mathfrak{A}_H(\mathbb{Z}[\alpha])$-module, proving Theorem \ref{thm:maintheorem}. Finally, we specialize to the cases in which $\mathcal{O}_L=\mathbb{Z}[\alpha]$.

\section{Preliminaries}

\subsection{Hopf-Galois extensions and Hopf-Galois structures}\label{sec:prelimhg}

A number field $L$ is said to be \textbf{Hopf-Galois} if there is a finite dimensional cocommutative $\mathbb{Q}$-Hopf algebra $H$ and a $\mathbb{Q}$-linear action $\cdot\colon H\otimes_{\mathbb{Q}}L\longrightarrow L$ such that:
\begin{itemize}
    \item $L$ is an $H$-module algebra with respect to $\cdot$ (see \cite[Chapter 2, Section 2.1, Page 39]{underwood2015}).
    \item The map $\iota\colon L\otimes_{\mathbb{Q}}H\longrightarrow\mathrm{End}_{\mathbb{Q}}(L)$ defined by $\iota(x\otimes h)(y)=x(h\cdot y)$ is a $\mathbb{Q}$-linear isomorphism.
\end{itemize}

In this situation, the pair $(H,\cdot)$ is called a \textbf{Hopf-Galois structure} on $L$, and we also say that $L$ is $H$-Galois. If $L$ is a Galois extension of $\mathbb{Q}$ with Galois group $G$, then $\mathbb{Q}[G]$ is naturally endowed with $\mathbb{Q}$-Hopf algebra structure, and together with its classical action on $L$ it gives rise to a Hopf-Galois structure on $L$, which is called the classical Galois structure.

We can determine all Hopf-Galois structures on a number field $L$ thank to Greither-Pareigis theory (see \cite{greitherpareigis}). Let $\widetilde{L}$ be the normal closure of $L$, $G=\mathrm{Gal}(\widetilde{L}/\mathbb{Q})$ and $G'=\mathrm{Gal}(\widetilde{L}/L)$. Let $X=G/G'$ be the set of left cosets of $G$ over $G'$.

\begin{thm}[Greither-Pareigis] The Hopf-Galois structures on $L$ are in one to one correspondence with the regular and $G$-stable subgroups of $\mathrm{Perm}(X)$.
\end{thm}

A subgroup $N\leq\mathrm{Perm}(X)$ is said to be regular if its action on $X$ is simply transitive (in particular, $|N|=|X|=[L:\mathbb{Q}]$). On the other hand, $N$ is $G$-stable if $\lambda(g)N\lambda(g^{-1})\subset N$ for all $g\in G$, where $\lambda\colon G\longrightarrow\mathrm{Perm}(X)$ is defined by $\lambda(g)(\overline{h})=\overline{gh}$. Given a regular and $G$-stable subgroup $N$ of $\mathrm{Perm}(X)$, its corresponding Hopf-Galois structure on $L$ has underlying Hopf algebra $$H=\widetilde{L}[N]^G=\{h\in\widetilde{L}[N]\,\mid\,g(h)=h\hbox{ for all }g\in G\},$$ where the action of $G$ on $\widetilde{L}[N]$ is given by the classical action on $\widetilde{L}$ and by conjugation with $\lambda(G)$ on $N$. The action of $H$ on $L$ is determined as follows: if $h=\sum_{i=1}^nh_i\eta_i\in H$ and $x\in L$, then \begin{equation}\label{eq:hopfaction}
h\cdot x=\sum_{i=1}^nh_i\eta_i^{-1}(\overline{1_G})(x).
\end{equation}

We say that the number field $L$ is almost classically Galois if $G'$ admits a normal complement $J$ within $G$, that is, $G=J\rtimes G'$ with $J$ a normal subgroup of $G$. In that case, $L$ is Hopf-Galois and a regular and $G$-stable subgroup of $\mathrm{Perm}(G)$ is given by $N=\lambda(J)$. Suppose that $J$ is abelian. Calling $M=\widetilde{L}^J$, the Hopf-Galois structure on $L/K$ corresponding to this subgroup is given by $H=M[\lambda(J)]^{G'}$, where the action is similar to the one above (see \cite[Proposition 2.9]{gilkummer}).

\subsection{Module structure of $\mathbb{Z}$-orders in a number field}

The classical problem in Galois module theory is to study the freeness of the ring of integers in a Galois number field as a module over its associated order in the Galois group algebra, defined as the maximal $\mathbb{Z}$-order in the Galois group algebra acting on the ring of integers. More generally, for a Hopf-Galois number field, the notion of associated order in a Hopf-Galois structure is naturally introduced, and there is the more general problem of studying the module structure of the ring of integers in such an associated order. We will work with a more general situation, consisting in considering $\mathbb{Z}$-orders of a Hopf-Galois number field, not just the ring of integers, and their module structure over a generalized notion of associated order. We will follow \cite{johnston2011}.

Let $K$ be a field which is the quotient field of a noetherian integral domain $R$. By \textbf{$R$-lattice} in a finite-dimensional $K$-algebra $A$ we mean a finitely generated $R$-submodule $M$ in $A$ such that $M$ contains a $K$-basis for $A$. Equivalently, an $R$-lattice in $A$ is a free finitely generated $R$-submodule of rank $n$ of $A$, where $n=\mathrm{dim}_K(A)$. An \textbf{$R$-order} in $A$ is an $R$-lattice of $A$ which in addition is a unitary subring of $A$. Given an $R$-lattice $M$ of $A$, the left order of $M$ in $A$, defined as $$\mathfrak{A}_l(A;M)=\{x\in A\,\mid\,x\cdot M\subseteq M\},$$ is an $R$-order of $A$ (see \cite[Proposition 3.12]{johnston2011}).

Now, let $L$ be an $H$-Galois number field and let $\mathcal{O}$ be a $\mathbb{Z}$-order of $L$. By \cite[(2.16)]{childs}, $L$ is free as an $H$-module of rank one, or equivalently, $L$ is isomorphic to $H$ as an $H$-module (note that this is a generalization of the normal basis theorem for Hopf-Galois extensions). Therefore, we can identify $\mathcal{O}$ with a $\mathbb{Z}$-lattice in $H$ and consider the left order $$\mathfrak{A}_H(\mathcal{O})\coloneqq\mathfrak{A}_l(H;\mathcal{O})=\{h\in H\,\mid\,h\cdot\mathcal{O}\subseteq\mathcal{O}\}.$$ This is a $\mathbb{Z}$-order of $A$. Arguing as in \cite[(12.5)]{childs}, one proves that if $\mathfrak{A}$ is a $\mathbb{Z}$-order in $H$ such that $\mathcal{O}$ is $\mathfrak{A}$-free, then $\mathfrak{A}=\mathfrak{A}_H(\mathcal{O})$.

If $\mathcal{O}=\mathcal{O}_L$ is the ring of integers of $L$, then $\mathfrak{A}_H(\mathcal{O}_L)\equiv\mathfrak{A}_H$ is called the associated order in $H$, and this is the usual object of study in Hopf-Galois module theory. In analogy, we will refer to $\mathfrak{A}_H(\mathcal{O})$ as the associated order of $\mathcal{O}$ in $H$. We consider the problem of finding a $\mathbb{Z}$-basis of $\mathfrak{A}_H(\mathcal{O})$ and a necessary and sufficient condition for $\mathcal{O}$ to be $\mathfrak{A}_H(\mathcal{O})$-free.

\section{An effective method to study the freeness of a $\mathbb{Z}$-order}\label{sec:redmethod}


Let $L$ be an $H$-Galois number field. Rio and the author \cite{gilrioquartic,gilrioinduced} introduced a method to study the $\mathfrak{A}_H$-freeness of the ring of integers $\mathcal{O}_L$ in $L$, which is well suited for extensions of low degree. In what follows we proceed to extend it to study the $\mathfrak{A}_H(\mathcal{O})$-freeness of a $\mathbb{Z}$-order $\mathcal{O}$ of $L$. By simplicity, we will work with extensions of number fields whose ground field is $\mathbb{Q}$, but the same development holds when the ground field is any number field $K$ whose ring of integers $\mathcal{O}_K$ is a PID.

Fix a $\mathbb{Q}$-basis $W=\{w_i\}_{i=1}^n$ of $H$ and a $\mathbb{Q}$-basis $B=\{\gamma_j\}_{j=1}^n$ of $L$. For each $1\leq i,j\leq n$, write $$w_i\cdot\gamma_j=\sum_{k=1}^nm_{ij}^{(k)}(H_W,L_B)\gamma_j,\quad m_{ij}^{(k)}(H_W,L_B)\in \mathbb{Q}.$$ Given $1\leq j\leq n$, denote $$M_j(H_W,L_B)=\begin{pmatrix}
|& | &\dots  &|  \\
w_1\cdot\gamma_j&w_2\cdot\gamma_j&\dots &w_n\cdot\gamma_j \\
|& |&\dots & |\\
\end{pmatrix}\in \mathcal{M}_n(\mathbb{Q}),$$ and let $$M(H_W,L_B)=\begin{pmatrix}
\\[-1ex]
M_1(H,L) \\[1ex] 
\hline\\[-1ex]
\cdots \\[1ex]
\hline \\[-1ex]
M_n(H,L)\\
\\[-1ex]
\end{pmatrix}\in \mathcal{M}_{n^2\times n}(\mathbb{Q}).$$ 

\begin{pro} There is a unimodular matrix $U\in\mathrm{GL}_{n^2}(\mathbb{Z})$ and a matrix $D\in\mathcal{M}_n(\mathbb{Q})$ such that \begin{equation}\label{eq:reducingM}
    UM(H_W,L_B)=\begin{pmatrix}D \\ \hline \\[-2ex] O\end{pmatrix},
\end{equation} where $O$ is the zero matrix of $\mathcal{M}_{(n^2-n)\times n}(\mathbb{Q})$.
\end{pro}
\begin{proof}
Let us show that the proof at \cite[Theorem 2.3]{gilrioquartic} for the case in which $B$ is an integral basis for $L$ also holds in this case. First, by \cite[Theorem 3.2]{kaplansky} we have that every $2\times 1$ matrix with coefficients in $\mathbb{Z}$ can be reduced by means of a unimodular matrix to a matrix of the same size with $0$ as lower entry. By \cite[Theorem 3.5]{kaplansky}, the statement holds for matrices with coefficients in $\mathbb{Z}$. Note that, by definition, we have $M(H_W,L_B)\in\mathcal{M}_{n\times n^2}(\mathbb{Q})$. Define $c$ to be the least common multiple of all denominators of entries in $M(H_W,L_B)$ and let $M=\frac{1}{c}M(H_W,L_B)\in\mathcal{M}_{n^2\times n}(\mathbb{Z})$, so that $M(H_W,L_B)=cM$. Let $U\in\mathrm{GL}_{n^2}(\mathbb{Z})$ and $D'\in\mathcal{M}_n(\mathbb{Z})$ be such that $$UM=\begin{pmatrix}D' \\ \hline \\[-2ex] O\end{pmatrix}.$$ Then the statement holds with $D=cD'$.
\end{proof}

\begin{defi} Any matrix $D\in\mathcal{M}_n(\mathbb{Q})$ fulfilling \eqref{eq:reducingM} is called a \textbf{reduced matrix} of $M(H_W,L_B)$.
\end{defi}

If the matrix $M(H_W,L_B)$ has integer coefficients, then the Hermite normal form of $M(H_W,L_B)$ is a reduced matrix.

In what follows, we will assume that $B$ is a $\mathbb{Z}$-basis for the $\mathbb{Z}$-order $\mathcal{O}$ of $L$ and we will show how to construct a $\mathbb{Z}$-basis of $\mathfrak{A}_H(\mathcal{O})$.

\begin{lem}\label{lem:usereduced} Assume that $B$ is a $\mathbb{Z}$-basis of $\mathcal{O}$ and let $D$ be a reduced matrix of $M(H_W,L_B)$. Given $h=\sum_{i=1}^nh_iw_i\in H$, $h\in\mathfrak{A}_H(\mathcal{O})$ if and only if $$D\begin{pmatrix}h_1 \\ \vdots \\ h_n\end{pmatrix}\in\mathbb{Z}^n.$$
\end{lem}
\begin{proof}
The matrix $M(H_W,L_B)$ is the matrix of the linear map $\rho_H\colon H\longrightarrow\mathrm{End}_{\mathbb{Q}}(L)$, where in $H$ we fix the $K$-basis $W$ and in $\mathrm{End}_{\mathbb{Q}}(L)$ we fix the canonical basis of its identification with $\mathcal{M}_{n^2}(\mathbb{Q})$ by means of the $\mathbb{Q}$-basis $B$ of $L$. Then, $h\in\mathfrak{A}_H(\mathcal{O})$ if and only if \begin{equation}\label{eqmembership}
    M(H_W,L_B)\begin{pmatrix}h_1 \\ \vdots \\ h_n\end{pmatrix}\in\mathbb{Z}^{n^2}.
\end{equation} By the definition of reduced matrix, there is a unimodular matrix $U$ such that $$UM(H_W,L_B)=\begin{pmatrix}D \\ \hline \\[-2ex] O\end{pmatrix}.$$ Since $U$ is unimodular, it sends $\mathbb{Z}^{n^2}$ to itself. Thus, applying $U$ to \eqref{eqmembership} yields that $h\in\mathfrak{A}_H(\mathcal{O})$ if and only if $$\begin{pmatrix}D \\ \hline \\[-2ex] O\end{pmatrix}\begin{pmatrix}h_1 \\ \vdots \\ h_n\end{pmatrix}\in\mathbb{Z}^{n^2}.$$ This is clearly equivalent to the condition in the statement.
\end{proof}

\begin{pro}\label{pro:basisassocord} Suppose that $B$ is a $\mathbb{Z}$-basis of $\mathcal{O}$. Let $D$ be a reduced matrix of $M(H_W,L_B)$ and call $D^{-1}=(d_{ij})_{i,j=1}^n$. The elements $$v_i=\sum_{l=1}^nd_{li}w_l,\,1\leq i\leq n$$ form a $\mathbb{Z}$-basis of $\mathfrak{A}_H(\mathcal{O})$.
\end{pro}
\begin{proof}
Let $h=\sum_{l=1}^nh_lw_l\in H$, with $h_l\in\mathbb{Q}$.  By Lemma \ref{lem:usereduced}, $h\in\mathfrak{A}_H(\mathcal{O})$ if and only if there exist elements $c_1,...,c_n\in\mathbb{Z}$ such that $$D\begin{pmatrix}
h_1 \\
\vdots \\
h_n
\end{pmatrix}=\begin{pmatrix}
c_1 \\
\vdots \\
c_n
\end{pmatrix}.$$ Multiplying by $D^{-1}$ at both sides, this is equivalent to $$\begin{pmatrix}
h_1 \\
\vdots \\
h_n
\end{pmatrix}=\begin{pmatrix}
d_{11} & \cdots & d_{1n} \\
\vdots & \ddots & \vdots \\
d_{n1} & \cdots & d_{nn}
\end{pmatrix}\begin{pmatrix}
c_1 \\
\vdots \\
c_n
\end{pmatrix},$$ that is, $$h_l=\sum_{i=1}^{n}d_{li}c_i,\,1\leq l\leq n.$$ Hence, $h\in\mathfrak{A}_H$ if and only if there exist $c_1,...,c_n\in\mathbb{Z}$ such that $$h=\sum_{l=1}^{n}\sum_{i=1}^{n}d_{li}c_iw_l=\sum_{i=1}^{n}c_i\left(\sum_{l=1}^{n}d_{li}w_l\right)=\sum_{i=1}c_iv_i.$$ The last member clearly belongs to $\langle v_1,...,v_n\rangle_{\mathbb{Z}}$. Hence, $\{v_i\}_{i=1}^n$ is a $\mathbb{Z}$-system of generators of $\mathfrak{A}_H(\mathcal{O})$. Now, it is $\mathbb{Q}$-linearly independent because $\{w_i\}_{i=1}^n$ is a $\mathbb{Q}$-basis and $D^{-1}$ is invertible, so it is also $\mathbb{Z}$-linearly independent and hence a $\mathbb{Z}$-basis of $\mathfrak{A}_H(\mathcal{O})$.
\end{proof}

From Proposition \ref{pro:basisassocord} it follows immediately that $D^{-1}$ is the change basis matrix from $\mathfrak{A}_H(\mathcal{O})$ to the $\mathbb{Z}$-lattice $\mathfrak{H}(\mathcal{O)}=\langle W\rangle_{\mathbb{Z}}$ generated by $W$, and conversely for the reduced matrix $D$. Therefore, $$[\mathfrak{A}_H(\mathcal{O}):\mathfrak{H}(\mathcal{O})]=|\mathrm{det}(D)|$$ for any reduced matrix $D$. We denote it by $I_W(\mathcal{O})$.

Next, we consider the problem of characterizing the $\mathfrak{A}_H(\mathcal{O})$-freeness of $\mathcal{O}$. Given $\beta=\sum_{j=1}^n\beta_j\gamma_j\in\mathcal{O}$, let $$M_{\beta}(H_W,L_B)\coloneqq\sum_{j=1}^n\beta_jM_j(H_W,L_B).$$

\begin{pro}\label{pro:criterionfreeness} Assume that $B$ is a $\mathbb{Z}$-basis of $\mathcal{O}$. An element $\beta\in\mathcal{O}$ is an $\mathfrak{A}_H(\mathcal{O})$-free generator of $\mathcal{O}$ if and only if \begin{equation}\label{eq:condfreeness}
    |\mathrm{det}(M_{\beta}(H_W,L_B))|=I_W(\mathcal{O}).
\end{equation} 
\end{pro}
\begin{proof}
The element $\beta\in\mathcal{O}_L$ is an $\mathfrak{A}_H(\mathcal{O})$-free generator of $\mathcal{O}_L$ if and only if $\mathcal{O}=\mathfrak{A}_H(\mathcal{O})\cdot\beta$. Now, by definition of $M_{\beta}(H_W,L_B)$, one has $$|\mathrm{det}(M_{\beta}(H_W,L_B))|=[\mathcal{O}:\mathfrak{H}(\mathcal{O})\cdot\beta]_{\mathbb{Z}}=[\mathcal{O}:\mathfrak{A}_H(\mathcal{O})\cdot\beta]_{\mathbb{Z}}I_W(\mathcal{O}).$$ The statement follows immediately.
\end{proof}

This gives rise to the following procedure to determine whether $\mathcal{O}_L$ is $\mathfrak{A}_H(\mathcal{O})$-free.

\begin{enumerate}
    \item\label{alg:step1} We assume that we know a $\mathbb{Z}$-basis $B$ of $\mathcal{O}$.
    \item\label{alg:step2} Fix a $\mathbb{Q}$-basis $W$ of $H$ and determine the coefficients of its action on $B$. From this, construct the matrix $M(H_W,L_B)$.
    \item\label{alg:step3} Transform integrally $M(H_W,L_B)$ to a reduced matrix $D\in\mathcal{M}_n(\mathbb{Q})$ and calculate $I_W(\mathcal{O})=|\mathrm{det}(D)|$.
    \item\label{alg:step4} Given $\beta\in\mathcal{O}_L$, compute $\mathrm{det}(M_{\beta}(H_W,L_B))$ and compare with $I_W(H,L)$.
\end{enumerate}

If we also want to find a $\mathbb{Z}$-basis of $\mathfrak{A}_H(\mathcal{O})$, we calculate the inverse $D^{-1}$ of the reduced matrix $D$ and follow Proposition \ref{pro:basisassocord}.

In order to ease the construction of $M(H_W,L_B)$ at \ref{alg:step2}, we introduce the following object.

\begin{defi}\label{def:grammatrix} The \textbf{Gram matrix} of the Hopf Galois structure $H$ is defined as the matrix $$G(H_W,L_B)=\begin{pmatrix}
w_1\cdot\gamma_1 & w_1\cdot\gamma_2 & \cdots & w_1\cdot\gamma_n \\
w_2\cdot\gamma_1 & w_2\cdot\gamma_2 & \cdots & w_2\cdot\gamma_n \\
\vdots & \vdots & \vdots & \vdots \\
w_n\cdot\gamma_1 & w_n\cdot\gamma_2 & \cdots & w_n\cdot\gamma_n \\
\end{pmatrix}\in\mathcal{M}_n(L).$$
\end{defi}

We have chosen this name in an analogy to the case of the Gram matrix of a scalar product. From the knowledge of $G(H_W,L_B)$ we can construct $M(H_W,L_B)$ by writing down the coordinates of the entries of the latter with respect to $B$ and arranging them accordingly in the former.

As for \ref{alg:step3}, recall that we need a unimodular matrix $U\in\mathrm{GL}_{n^2}(\mathbb{Z})$ as in \ref{eq:reducingM}, in order to find a reduced matrix $D$. The matrix $U$ is actually the product of the matrices corresponding to elementary transformations on the rows of $M(H,L)$, which need to be themselves unimodular. Hence, the only allowed operations are: swap rows, add a row an integer multiple of another one, and switch the sign of a row.

Finally, regarding \ref{alg:step4}, we will see the condition \eqref{eq:condfreeness} as a system of equations where the variables are the coordinates $(\beta_j)_{j=1}^n$ of $\beta$ with respect to $B$.

\section{Extended Euclid's algorithm and continued fractions}\label{sec:euclidalg}

We shall view in this section some techniques addressed to find the Hermite normal form of matrices with integer parameters, which will be eventually needed to find reduced matrices. 

Let $x,y\in\mathbb{Z}$ with $y\neq0$. Suppose that we want to find the Hermite normal form of a matrix of the form $$A=\begin{pmatrix}
x & \lambda \\
y & \gamma
\end{pmatrix}$$ for other two integers $\lambda,\gamma\in\mathbb{Z}$. In order to reduce the first column, we apply Euclid's algorithm to $x$ and $y$. Write $r_{-1}\coloneqq x$ and $r_0\coloneqq y$, and $r_i=a_{i+1}r_{i+1}+r_{i+2}$ for each $-1\leq i\leq n-1$, where $r_n=\mathrm{gcd}(x,y)$ and $r_{n+1}=0$, so that the last division is $r_{n-1}=a_nr_n$. Let us define two finite sequences $\{\mu_i\}_{i=0}^{n+1}$, $\{\nu_i\}_{i=0}^{n+1}$ recursively as follows:
\vspace{2mm}

\begin{minipage}{0.5 \textwidth}
    $$\begin{cases}
        \mu_0=0,\,\mu_1=1,\\
        \mu_i=-a_{i-1}\mu_{i-1}+\mu_{i-2},\,i\geq2
    \end{cases}$$
\end{minipage}
\begin{minipage}{0.5 \textwidth}
    $$\begin{cases}
        \nu_0=1,\,\nu_1=-a_0,\\
        \nu_i=-a_{i-1}\nu_{i-1}+\nu_{i-2},\,i\geq2
    \end{cases}$$
\end{minipage}

\begin{pro} For each $k\in\{1,2\}$, let $F_k$ denote the $k$-th row of the matrix $A$. Let us consider the following sequence of linear elementary transformations on $A$: $$F_1\mapsto F_1-a_0F_2,\quad F_2\mapsto F_2-a_1F_1,\quad F_1\mapsto F_1-a_2F_2,\dots$$ Then, for each $1\leq i\leq n+1$, the matrix we obtain after applying the transformation with $a_{i-1}$ is, up to ordering of rows, $$\begin{pmatrix}
    r_{i-1} & \mu_{i-1}\lambda+\nu_{i-1}\gamma \\
    r_i & \mu_i\lambda+\nu_i\gamma
\end{pmatrix}.$$
\end{pro}
\begin{proof}
We prove it by induction on $i$. For $i=1$, let us apply on $A$ the transformation $F_1\mapsto F_1-a_0F_2$: $$\begin{pmatrix}
x & \lambda \\
y & \gamma
\end{pmatrix}\longrightarrow\begin{pmatrix}
    x-a_0y & \lambda-a_0\gamma \\
    y & \gamma
\end{pmatrix}=\begin{pmatrix}
    r_1 & \mu_1\lambda+\nu_1\gamma \\
    r_0 & \mu_0\lambda+\nu_0\gamma
\end{pmatrix},$$ which coincides with the matrix at the statement up to ordering of the rows.

Let us assume that the statement holds for some $0\leq i\leq n$, so that after the $i$-th transformation we have (switching rows if necessary) $$\begin{pmatrix}
    r_{i-1} & \mu_{i-1}\lambda+\nu_{i-1}\gamma \\
    r_i & \mu_i\lambda+\nu_i\gamma
\end{pmatrix}.$$ Now, we perform the transformation $F_1\mapsto F_1-a_iF_2$, obtaining $$\begin{pmatrix}
    r_{i-1}-a_ir_i & \mu_{i-1}\lambda+\nu_{i-1}\gamma-a_i(\mu_i\lambda+\nu_i\gamma) \\
    r_i & \mu_i\lambda+\nu_i\gamma
\end{pmatrix}.$$ The $(1,1)$-entry is clearly $r_{i+1}$, while the $(1,2)$-entry can be rewritten as $$(-a_i\mu_i+\mu_{i-1})\lambda+(-a_i\nu_i+\nu_{i-1})\gamma=\mu_{i+1}\lambda+\nu_{i+1}\nu.$$ Thus, switching rows, we obtain the matrix $$\begin{pmatrix}
    r_i & \mu_i\lambda+\nu_i\gamma \\
    r_{i+1} & \mu_{i+1}\lambda+\nu_{i+1}\gamma
\end{pmatrix}.$$
\end{proof}

Specializing for $i=1$ in this result, we obtain the Hermite normal form of $A$.

\begin{coro}\label{coro:hermiteform} The Hermite normal form of the matrix $A$ is $$\begin{pmatrix}
    r_n & \mu_n\lambda+\nu_n\gamma \\
    0 & \mu_{n+1}\lambda+\nu_{n+1}\gamma
\end{pmatrix}.$$
\end{coro}

Next, we describe how to write the remainders $r_i$ in terms of the sequences $\{\mu_i\}$ and $\{\nu_i\}$, deriving a Bezout identity for $\mathrm{gcd}(x,y)$ in terms of $x$ and $y$.

\begin{pro} Given $0\leq i\leq n+1$, $$r_i=\mu_ix+\nu_iy.$$
\end{pro}
\begin{proof}
    We prove it by generalized induction on $i$. For $i=0$, we have $$\mu_0x+\nu_0y=y=r_0.$$ For $i=1$, we get $$\mu_1x+\nu_1y=x-a_0y=r_1.$$ Suppose that for some $0\leq i\leq n$ we have that $r_{i'}=\mu_{i'}x+\nu_{i'}y$ for all $0\leq i'\leq i$. Then $$r_{i+1}=r_{i-1}-a_ir_i=-a_i(\mu_ix+\nu_iy)+\mu_{i-1}x+\nu_{i-1}y=(-a_i\mu_i+\mu_{i-1})x+(-a_i\nu_i+\nu_{i-1})y=\mu_{i+1}x+\nu_{i+1}y.$$
\end{proof}

Specializing for $i=n$ and $i=n+1$ we obtain the following relations:

\begin{coro}\label{coro:relgcd0} With the previous notation, we have:
\begin{enumerate}
    \item\label{coro:relgcd} $r_n=\mu_nx+\nu_ny$.
    \item\label{coro:rel0} $0=\mu_{n+1}x+\nu_{n+1}y$.
\end{enumerate}
\end{coro}

Write $$\frac{x}{y}=[a_0;a_1,\dots,a_n]\coloneqq a_0+\cfrac{1}{a_1+\cfrac{1}{a_2+\cfrac{1}{\ddots+\cfrac{1}{a_n}}}}.$$ for the continued fraction expansion of $\frac{x}{y}$. For each $0\leq i\leq n$, let $$\frac{p_i}{q_i}=[a_0;a_1,\dots,a_i]$$ be the $i$-th convergent of $\frac{x}{y}$ in irreducible form, i.e. $\mathrm{gcd}(p_i,q_i)=1$. Note that, by definition, $p_n=\frac{x}{r_n}$ and $q_n=\frac{y}{r_n}$, since $r_n=\mathrm{gcd}(x,y)$. One checks easily that the sequences $\{p_i\}_{i=0}^n$ and $\{q_i\}_{i=0}^n$ can be expressed recursively as follows:
\begin{align*}
    \begin{cases}
        p_0=a_0,\,p_1=a_0a_1+1,\\
        p_i=a_ip_{i-1}+p_{i-2},\,i\geq2
    \end{cases} &&
    \begin{cases}
        q_0=1,\,q_1=a_1,\\
        q_i=a_iq_{i-1}+q_{i-2},\,i\geq2
    \end{cases}
\end{align*}

Actually, these sequences coincide with the previous ones up to sign.

\begin{pro}\label{pro:relseq} Given $1\leq i\leq n+1$, we have:
\begin{enumerate}
    \item $\mu_i=(-1)^{i-1}q_{i-1}$.
    \item $\nu_i=(-1)^ip_{i-1}$.
\end{enumerate}
\end{pro}
\begin{proof}
We prove both statements by generalized induction on $i$.

\begin{enumerate}
    \item Replacing $i=1$ at the right side gives $q_0=0=\mu_0$, while replacing $i=2$ yields $-q_1=-a_1=\mu_2$. Fix $0\leq i\leq n$ and suppose that $\mu_{i'}=(-1)^{i'-1}q_{i'-1}$ for all $0\leq i'\leq i$. Then $$\mu_{i+1}=-a_i\mu_i+\mu_{i-1}=-a_i(-1)^{i-1}q_{i-1}+(-1)^{i-2}q_{i-2}=(-1)^i(a_iq_{i-1}+q_{i-2})=(-1)^i\mu_i.$$
    \item For $i=1$, we obtain $-p_0=-a_0=\nu_1$. For $i=2$, we get $p_1=a_0a_1+1=\nu_2$. Fix $0\leq i\leq n$ and suppose that $\nu_{i'}=(-1)^{i'}p_{i'-1}$ for all $0\leq i'\leq i$. Then $$\nu_{i+1}=-a_i\nu_i+\nu_{i-1}=-a_i(-1)^ip_{i-1}+(-1)^{i-1}p_{i-2}=(-1)^{i+1}(a_ip_{i-1}+p_{i-2})=(-1)^{i+1}p_i.$$
\end{enumerate}
\end{proof}

\begin{coro}\label{coro:reln+1} With the previous notations, we have:
\begin{enumerate}
    \item\label{coro:reln+1a} $\mu_{n+1}=(-1)^n\frac{y}{r_n}$.
    \item\label{coro:reln+1b} $\nu_{n+1}=(-1)^{n+1}\frac{x}{r_n}$.
\end{enumerate}
\end{coro}

\section{Cubic number fields}\label{sec:cubicfields}

Let $L$ be a cubic number field. It can be checked that $L=\mathbb{Q}(\alpha)$, where $\alpha\in L$ is a root of a polynomial of the form \begin{equation}\label{eq:polynf}
    f(x)=x^3-ax+b,
\end{equation} where $a,b\in\mathbb{Z}$ and $v_p(a)<2$ or $v_p(a)<3$ for every prime $p$ (see for example \cite{llorentenart}). The discriminant of $f$ is $\Delta=4a^3-27b^2$.

\subsection{The Hopf-Galois structure on a cubic number field}

A cubic number field $L$ is always Hopf-Galois. Indeed, from \cite[(7.5)]{childs} we know that a prime degree extension is Hopf-Galois if and only if the Galois group of its normal closure over the ground field is solvable. In our case, calling $\widetilde{L}$ the normal closure of $L$ and $G=\mathrm{Gal}(\widetilde{L}/\mathbb{Q})$, $L$ is Hopf-Galois if and only if $G$ is solvable. This always holds, since $G\cong C_3$ if $L$ is Galois and $G\cong D_3$ otherwise. On the other hand, Byott \cite{byottuniqueness} proved that a Hopf-Galois extension whose degree $n$ is Burnside (that is, $n$ is coprime with its image by the Euler totient function) admits a unique Hopf-Galois structure. It is immediate that a prime number is Burnside, so our cubic number field $L$ admits a unique Hopf-Galois structure $H$. In the case that $L$ is Galois, its only Hopf-Galois structure is the classical one: $\mathbb{Q}[G]$ together with its classical action.

Let us describe explicitly the Hopf-Galois structure on $L/\mathbb{Q}$ in the general case. We will do so by using Greither-Pareigis theorem. The normal closure of $L$ is given by $\widetilde{L}=L(z)$, where $z\coloneqq\sqrt{\Delta}=\sqrt{4a^3-27b^2}$. Call $M\coloneqq\mathbb{Q}(z)$. Since $L\cap M=\mathbb{Q}$, $L$ and $M$ are $\mathbb{Q}$-linearly disjoint fields and $\widetilde{L}=LM$.

Let $\alpha_1\coloneqq\alpha$, $\alpha_2$, $\alpha_3$ be the roots of the polynomial $f$ as in \eqref{eq:polynf}, and let $\sigma\in G$ be defined by $\sigma=(\alpha_1,\alpha_2,\alpha_3)$ as a permutation of the roots of $f$. If $G\cong C_3$, then $$G=\langle\sigma\,\mid\,\sigma^3=\mathrm{Id}\rangle.$$ Otherwise, we have that $z\in\widetilde{L}-L$ and $\sigma(z)=z$. Let $\tau\in G$ be defined by $\tau(\alpha_i)=\alpha_i$, $i\in\{1,2,3\}$, and $\tau(z)=-z$. It follows that $$G=\langle\sigma,\tau\,\mid\,\sigma^3=\tau^2=\mathrm{Id},\,\tau\sigma=\sigma^2\tau\rangle.$$

Let $G'\coloneqq\mathrm{Gal}(\widetilde{L}/L)$. Necessarily, $G'=\langle\tau\rangle$. This subgroup admits a normal complement within $G$, namely $J=\langle\sigma\rangle$, so $L$ is almost classically Galois. Hence, its only Hopf-Galois structure is given by $H=M[\lambda(J)]^{G'}$, where $G'$ acts on $M$ by the classical action and on $J$ by conjugation with $\lambda(G')$ (as already mentioned at the end of Section \ref{sec:prelimhg}). Let us identify $J$ with $\lambda(J)$. Then, the previous action becomes simply the conjugacy action of $G'$ on $J$.

\begin{pro}\label{pro:basisH} Let $L$ be a cubic number field and let $H$ be its Hopf-Galois structure. Let $\sigma$ and $z$ be defined as above. A $\mathbb{Q}$-basis for the underlying $\mathbb{Q}$-Hopf algebra at $H$ is given by \begin{equation}\label{eq:basisH}
w_1=\mathrm{Id},\quad w_2=z(\sigma-\sigma^2),\quad w_3=\sigma+\sigma^2.
\end{equation}
\end{pro}
\begin{proof}
If $L/\mathbb{Q}$ is Galois, we know that $H=\mathbb{Q}[G]$, which has $\mathbb{Q}$-basis $\{\mathrm{Id},\sigma,\sigma^2\}$. Now, the matrix of coordinates of the $w_i$ with respect to this one is $$\begin{pmatrix}
    1 & 0 & 0 \\
    0 & z & 1 \\
    0 & -z & 1
\end{pmatrix},$$ which has non-zero determinant, so $\{w_1,w_2,w_3\}$ is a $\mathbb{Q}$-basis of $H$.

Suppose that $L/\mathbb{Q}$ is not Galois. It is trivial that $w_i\in M[J]$ for all $i\in\{1,2,3\}$. In order to check that they are fixed by the action of $G'$, it is enough to do so for the generator $\tau$. We have: $$\tau\cdot w_1=\tau\mathrm{Id}\tau^{-1}=\mathrm{Id},$$ $$\tau\cdot w_2=\tau(z)(\tau\sigma\tau^{-1}-\tau\sigma^2\tau^{-1})=-z(\sigma^2-\sigma)=z(\sigma-\sigma^2)=w_2,$$ $$\tau\cdot w_3=\tau\sigma\tau^{-1}+\tau\sigma^2\tau^{-1}=\sigma^2+\sigma=w_3.$$ Then the span $\langle w_1,w_2,w_3\rangle_{\mathbb{Q}}$ is contained in $H$, and since both have $\mathbb{Q}$-dimension $3$, they coincide.
\end{proof}

\section{Associated orders of monogenic orders in cubic number fields}\label{sec:assocorders}

Let $L$ be a cubic number field and write $\alpha$ for a primitive element of $L/\mathbb{Q}$ which is a root of a polynomial as in \eqref{eq:polynf}. We consider the following problem.

\begin{prob} Let $H$ be the only Hopf-Galois structure on $L$. Find a $\mathbb{Z}$-basis of $\mathfrak{A}_H(\mathbb{Z}[\alpha])$ and a necessary and sufficient condition for $\mathcal{O}$ being $\mathfrak{A}_H(\mathcal{O})$-free.
\end{prob}

We shall follow the method at Section \ref{sec:redmethod}. Of course, $B=\{1,\alpha,\alpha^2\}$ is a $\mathbb{Z}$-basis for $\mathcal{O}$. Whenever $B$ is an integral basis for $L$, $\mathcal{O}=\mathcal{O}_L$ and we recover the usual problem for the module structure of the ring of integers.

\subsection{Looking for the Gram matrix}

In this section we will find the Gram matrix $G(H_W,L_B)$ as defined in Definition \ref{def:grammatrix}.

\begin{pro}\label{pro:grammatrixcanon} With the previous notations, we have $$G(H_W,L_B)=\begin{pmatrix}
    1 & \alpha & \alpha^2 \\
    0 & 6a\alpha^2+9b\alpha-4a^2 & -9b\alpha^2-2a^2\alpha+6ab \\
    2 & -\alpha & -\alpha^2+2a
\end{pmatrix}.$$
\end{pro}
\begin{proof}
The first row of $G(H_W,L_B)$ corresponds to the action of $w_1=\mathrm{Id}$ on $B$, which is trivial. On the other hand, for the first column we have $$w_2\cdot1=z(1-1)=0,\quad w_3\cdot1=1+1=2.$$ 

Recall that $\sigma$ permutes the roots $\alpha_1=\alpha$, $\alpha_2$, $\alpha_3$ as the permutation $(\alpha_1,\alpha_2,\alpha_3)$. On the other hand, the symmetric functions of the roots give the relations $$\begin{cases}
    \alpha_1+\alpha_2+\alpha_3=0, \\
    \alpha_1\alpha_2+\alpha_1\alpha_3+\alpha_2\alpha_3=-a, \\
    \alpha_1\alpha_2\alpha_3=b.
\end{cases}$$ 

We first determine the action of $w_3$: $$w_3\cdot\alpha=\alpha_2+\alpha_3=-\alpha,$$ $$w_3\cdot\alpha^2=\alpha_2^2+\alpha_3^2.$$ At this point, we note that $$\alpha_1^2+\alpha_2^2+\alpha_3^2=(\alpha_1+\alpha_2+\alpha_3)^2-2(\alpha_1\alpha_2+\alpha_1\alpha_3+\alpha_2\alpha_3)=2a.$$ Therefore, $$w_3\cdot\alpha^2=-\alpha^2+2a,$$ as stated.

In order to determine the action of $w_2$, we need further preparations. Using Ruffini algorithm, we write $$f(x)=(x-\alpha)(x^2+\alpha x+\alpha^2-a).$$ The right side member is a quadratic polynomial in $L[x]$ with roots $\alpha_2$ and $\alpha_3$. Due to the symmetry between these roots, we can assume without loss of generality that $$\alpha_2=\frac{-\alpha+\sqrt{d}}{2},\quad\alpha_3=\frac{-\alpha-\sqrt{d}}{2},$$ where $d\coloneqq-3\alpha^2+4a$. Using mathematical software, one can check that $$\sqrt{d}=\frac{1}{z}(6a\alpha^2+9b\alpha-4a^2),$$ where the sign is chosen arbitrarily (the opposite sign would involve exchanging $\alpha_2$ and $\alpha_3$).

Finally, we calculate $$w_2\cdot\alpha=z(\alpha_2-\alpha_3)=z\sqrt{d}=6a\alpha^2+9b\alpha-4a^2,$$ $$w_2\cdot\alpha^2=z(\alpha_2^2-\alpha_3^2)=z(\alpha_2-\alpha_3)(\alpha_2+\alpha_3)=-\alpha(6a\alpha^2+9b\alpha-4a^2)=-9b\alpha^2-2a^2\alpha+6ab.$$
\end{proof}

\subsection{Reducing matrices of the action}

From the knowledge of $G(H_W,L_B)$ we construct $$M(H_W,L_B)=\begin{pmatrix}
    1 & 0 & 2 \\
    0 & 0 & 0 \\
    0 & 0 & 0 \\
    0 & -4a^2 & 0 \\
    1 & 9b & -1 \\
    0 & 6a & 0 \\
    0 & 6ab & 2a \\
    0 & -2a^2 & 0 \\
    1 & -9b & -1
\end{pmatrix}.$$ We aim to find a reduced matrix of this one. To do so, recall that we must reduce $M(H_W,L_B)$ using only row swapping, addition of integer multiple of rows and sign switching of rows. With these operations, it is easy to reduce $M(H_W,L_B)$ to the matrix \begin{equation}\label{eq:redmatrix1}
    \begin{pmatrix}
    1 & 0 & 2 \\
    0 & 2a^2 & 0 \\
    0 & 6a & 0 \\
    0 & 9b & -3 \\
    0 & 0 & 2a \\
    0 & 0 & 6
\end{pmatrix}.
\end{equation} Now, taking into account Euclid's algorithm on $(2a^2,6a)$ and $(2a,6)$, this matrix reduces to \begin{equation}\label{eq:redmatrix2}
    \begin{pmatrix}
    1 & 0 & 2 \\
    0 & g_1 & 0 \\
    0 & 9b & -3 \\
    0 & 0 & g_2
\end{pmatrix},
\end{equation} where $g_1=\mathrm{gcd}(2a^2,6a)$ and $g_2=\mathrm{gcd}(2a,6)$. It is easy to check that $$g_1=\begin{cases}
    2a & \hbox{if }3\nmid a, \\
    6a & \hbox{if }3\mid a,
\end{cases}$$ $$g_2=\begin{cases}
    2 & \hbox{if }3\nmid a, \\
    6 & \hbox{if }3\mid a
\end{cases}.$$

\subsection*{Case 1: $3\nmid a$}

If $3\nmid a$, \eqref{eq:redmatrix2} becomes \begin{equation}\label{eq:redmatrix3}
    \begin{pmatrix}
    1 & 0 & 0 \\
    0 & 2a & 0 \\
    0 & 9b & 1 \\
    0 & 0 & 2
\end{pmatrix}.
\end{equation} Let us call $g=\mathrm{gcd}(a,b)$ and $h=\mathrm{gcd}(2a,9b)$. Let us find $h$ in terms of $g$. It is clear that $g$ divides $h$. Since $3\nmid a$, $h=\mathrm{gcd}(2a,b)$.

If $v_2(a)\geq v_2(b)$, we have that $v_2(2a)>v_2(b)$, so $v_2(h)=v_2(b)=v_2(g)$. Otherwise, if $v_2(a)<v_2(b)$, then $v_2(2a)\leq v_2(b)$ and $v_2(h)=v_2(2a)=v_2(a)+1=v_2(g)+1$. On the other hand, if $p$ is an odd prime, we have that $v_p(2a)=v_p(a)$ and $v_p(h)=v_p(g)$. We obtain that $$h=\begin{cases}
    g & \hbox{if }v_2(a)\geq v_2(b), \\
    2g & \hbox{if }v_2(a)<v_2(b).
\end{cases}$$

\begin{pro}\label{pro:redanotdiv3} Assume that $3\nmid a$. A reduced matrix for $M(H_W,L_B)$ is:
\begin{enumerate}
    \item If $v_2(a)\geq v_2(b)$, $$\begin{pmatrix}
        1 & 0 & 0 \\
        0 & g & 1 \\
        0 & 0 & 2
    \end{pmatrix}.$$

    \item If $v_2(a)<v_2(b)$, $$\begin{pmatrix}
        1 & 0 & 0 \\
        0 & 2g & 0 \\
        0 & 0 & 1
    \end{pmatrix}.$$
\end{enumerate}
\end{pro}
\begin{proof}
We apply Euclid's algorithm with $x=9b$ and $y=2a$. By Corollary \ref{coro:hermiteform}, \eqref{eq:redmatrix3} can be transformed integrally into \begin{equation}\label{eq:redmatrix4}
    \begin{pmatrix}
    1 & 0 & 0 \\
    0 & h & \mu_n \\
    0 & 0 & \mu_{n+1} \\
    0 & 0 & 2
\end{pmatrix}.
\end{equation}

Assume that $v_2(a)\geq v_2(b)$. Then $h=g$ and $v_2(h)=v_2(b)=v_2(x)$. 
From Corollary \ref{coro:relgcd0} \ref{coro:rel0}, we have $\mu_{n+1}x=-\nu_{n+1}y$. Taking $2$-valuations we obtain $$v_2(\mu_{n+1})+v_2(b)=v_2(\nu_{n+1})+v_2(a)+1.$$ Then, $$v_2(\mu_{n+1})=v_2(\nu_{n+1})+v_2(a)+1-v_2(b)>0$$ because $v_2(a)\geq v_2(b)$. Hence, $\mu_{n+1}$ is even. On the other hand, from Corollary \ref{coro:relgcd0} we know that $h=\mu_nx+\nu_ny$. Since $h=g$, $v_2(\mu_nx+\nu_ny)=v_2(g)=v_2(b)$. Now, $v_2(\mu_nx)=v_2(\mu_n)+v_2(b)$ and $v_2(\nu_ny)=v_2(\nu_n)+v_2(a)+1$. If $\mu_n$ is even, we obtain that both $\mu_nx$ and $\nu_ny$ have $2$-valuation strictly greater than $v_2(b)$, contradicting that their sum has $2$-valuation exactly equal to $v_2(b)$. Therefore, $\mu_n$ is odd. Since $\mu_n\equiv1\,(\mathrm{mod}\,2)$ and $\mu_{n+1}\equiv0\,(\mathrm{mod}\,2)$, we can reduce immediately \eqref{eq:redmatrix4} to $$\begin{pmatrix}
        1 & 0 & 0 \\
        0 & g & 1 \\
        0 & 0 & 2
    \end{pmatrix}.$$

Now, suppose that $v_2(a)<v_2(b)$. Then $h=2g$ and $v_2(x)=v_2(b)\geq v_2(a)+1=v_2(y)$, so $v_2(h)=v_2(y)$ and then $\frac{y}{h}$ is odd. From Corollary \ref{coro:reln+1} \ref{coro:reln+1a} we obtain that $\mu_{n+1}$ is odd. Then, \eqref{eq:redmatrix4} reduces to $$\begin{pmatrix}
        1 & 0 & 0 \\
        0 & 2g & 0 \\
        0 & 0 & 1
    \end{pmatrix}.$$
\end{proof}

Applying Proposition \ref{pro:basisassocord}, we obtain the following.

\begin{coro} Suppose that $3\nmid a$. Then $I_W(\mathbb{Z}[\alpha])=2g$ and a $\mathbb{Z}$-basis of $\mathfrak{A}_H(\mathbb{Z}[\alpha])$ at each case is as follows.

\begin{center}
\begin{tabular}{|c|c|c|} \hline
    \hbox{Case} & $I_W(\mathbb{Z}[\alpha])$ & $\mathbb{Z}$-basis of $\mathfrak{A}_H(\mathbb{Z}[\alpha])$ \\ \hline
    $v_2(a)\geq v_2(b)$ & $2g$ & $\{w_1,\frac{w_2}{g},\frac{-w_2+gw_3}{2g}\}$ \\ \hline
    $v_2(a)<v_2(b)$ & $2g$ & $\{w_1,\frac{w_2}{2g},w_3\}$ \\ \hline
\end{tabular}
\end{center}
\end{coro}

\subsection*{Case 2: $3\mid a$}

If $3\mid a$, \eqref{eq:redmatrix2} becomes \begin{equation}\label{eq:redmatrix5}
    \begin{pmatrix}
    1 & 0 & 2 \\
    0 & 6a & 0 \\
    0 & 9b & 3 \\
    0 & 0 & 6
\end{pmatrix}.
\end{equation} As before, let us call $g=\mathrm{gcd}(a,b)$ and let us determine $h=\mathrm{gcd}(6a,9b)$ in terms of $g$.

\begin{pro}\label{pro:valuesh} With the previous notation, $$h=\begin{cases}
    3g & \hbox{if }v_2(a)\geq v_2(b)\hbox{ and }v_3(a)\leq v_3(b). \\
    6g & \hbox{if }v_2(a)<v_2(b)\hbox{ and }v_3(a)\leq v_3(b). \\
    9g & \hbox{if }v_2(a)\geq v_2(b)\hbox{ and }v_3(a)>v_3(b). \\
    18g & \hbox{if }v_2(a)<v_2(b)\hbox{ and }v_3(a)>v_3(b). \\
\end{cases}$$
\end{pro}
\begin{proof}
If $p$ is a prime with $p\notin\{2,3\}$, then $v_p(a)=v_p(6a)$ and $v_p(b)=v_p(9b)$, so $v_p(h)=v_p(g)$. Hence, it is enough to analyze $v_p(g)$ and $v_p(h)$ with $p\in\{2,3\}$.

If $v_2(a)\geq v_2(b)$, then $v_2(6a)>v_2(9b)$, so $$v_2(h)=v_2(9b)=v_2(b)=v_2(g).$$ Otherwise, if $v_2(a)<v_2(b)$, then $v_2(6a)\leq v_2(9b)$, so $$v_2(h)=v_2(6a)=v_2(a)+1=v_2(g)+1.$$

On the other hand, if $v_3(a)\leq v_3(b)$, we have that $v_3(6a)<v_3(9b)$, so $$v_3(h)=v_3(6a)=v_3(a)+1=v_3(g)+1.$$ Otherwise, if $v_3(a)>v_3(b)$, then $v_3(6a)\geq v_3(9b)$, so $$v_3(h)=v_3(9b)=v_3(b)+2=v_3(g)+2.$$

Combining these possibilities gives the statement.
\end{proof}

\begin{pro} Assume that $3\mid a$. A reduced matrix for $M(H_W,L_B)$ is as follows:
\begin{itemize}
    \item[1.] If $v_2(a)\geq v_2(b)$ and $v_3(a)\leq v_3(b)$, $$\begin{pmatrix}
        1 & 0 & 2 \\
        0 & 3g & 3 \\
        0 & 0 & 6
    \end{pmatrix}.$$

    \item[2.] If $v_2(a)<v_2(b)$ and $v_3(a)\leq v_3(b)$, $$\begin{pmatrix}
        1 & 0 & 2 \\
        0 & 6g & 0 \\
        0 & 0 & 3
    \end{pmatrix}.$$

    \item[3.] If $v_2(a)\geq v_2(b)$ and $v_3(a)>v_3(b)$, $$\begin{pmatrix}
        1 & 0 & 2 \\
        0 & 9g & 3 \\
        0 & 0 & 6
    \end{pmatrix}.$$

    \item[4.] If $v_2(a)<v_2(b)$ and $v_3(a)>v_3(b)$, $$\begin{pmatrix}
        1 & 0 & 2 \\
        0 & 18g & 0 \\
        0 & 0 & 3
    \end{pmatrix}.$$
\end{itemize}
\end{pro}
\begin{proof}
We apply Euclid's algorithm with $x=9b$ and $y=6a$. By Corollary \ref{coro:hermiteform}, \ref{eq:redmatrix5} can be transformed integrally into \begin{equation}\label{eq:redmatrix6}
    \begin{pmatrix}
    1 & 0 & 2 \\
    0 & h & 3\mu_n \\
    0 & 0 & 3\mu_{n+1} \\
    0 & 0 & 6
\end{pmatrix}.
\end{equation}
Assume that $v_2(a)\geq v_2(b)$. Arguing as in the analogous case in the proof of Proposition \ref{pro:redanotdiv3}, we prove that $\mu_{n+1}$ is even. On the other hand, Corollary \ref{coro:relgcd0} \ref{coro:relgcd}, we have $h=\mu_nx+\nu_ny$. By Proposition \ref{pro:valuesh}, $h=3g$ if $v_3(a)\geq v_3(b)$ and $h=9g$ if $v_3(a)>v_3(b)$, so $v_2(h)=v_2(g)=v_2(b)$ in both cases. Then, $v_2(b)=v_2(\mu_n9b+\nu_n6a)$ with $v_2(\mu_n9b)=v_2(\mu_n)+v_2(b)$ and $v_2(\nu_n6a)=v_2(\nu_n)+v_2(a)+1>v_2(b)$. Necessarily, $v_2(\mu_n)=0$, that is, $\mu_n$ is odd. Then \eqref{eq:redmatrix6} reduces to $$\begin{pmatrix}
        1 & 0 & 2 \\
        0 & h & 3 \\
        0 & 0 & 6
    \end{pmatrix}.$$

Now, suppose that $v_2(a)<v_2(b)$. Arguing as in the analogous case in proof of Proposition \ref{pro:redanotdiv3}, we obtain that $\mu_{n+1}$ is odd, and then \eqref{eq:redmatrix6} reduces to $$\begin{pmatrix}
        1 & 0 & 2 \\
        0 & h & 0 \\
        0 & 0 & 3
    \end{pmatrix}.$$ The proof is finished by taking into account the values of $h$ at Proposition \ref{pro:valuesh}.
\end{proof}

\begin{coro} Suppose that $B$ is an integral basis and $3\mid a$. The index $I_W(\mathbb{Z}[\alpha])$ and a $\mathbb{Z}$-basis of $\mathfrak{A}_H(\mathbb{Z}[\alpha])$ are as follows: 

\begin{center}
\begin{tabular}{|c|c|c|} \hline
    \hbox{Case} & $I_W(\mathbb{Z}[\alpha])$ & $\mathbb{Z}$-basis of $\mathfrak{A}_H(\mathbb{Z}[\alpha])$ \\ \hline
    $v_2(a)\geq v_2(b)$ and $v_3(a)\leq v_3(b)$ & $18g$ & $\{w_1,\frac{w_2}{3g},\frac{-2gw_1-w_2+gw_3}{6}\}$ \\ \hline
    $v_2(a)<v_2(b)$ and $v_3(a)\leq v_3(b)$ & $18g$ & $\{w_1,\frac{w_2}{6g},\frac{-2w_1+w_3}{3}\}$ \\ \hline
    $v_2(a)\geq v_2(b)$ and $v_3(a)>v_3(b)$ & $54g$ & $\{w_1,\frac{w_2}{9g},\frac{-6gw_1-w_2+3gw_3}{18g}\}$ \\ \hline
    $v_2(a)<v_2(b)$ and $v_3(a)>v_3(b)$ & $54g$ & $\{w_1,\frac{w_2}{18g},\frac{-2w_1+w_3}{3}\}$ \\ \hline
\end{tabular}
\end{center}

\end{coro}

Note that the inequality $v_3(a)>v_3(b)$ implies that $3\mid a$, so the latter is redundant. Hence, we have obtained $$I_W(\mathbb{Z}[\alpha])=\begin{cases}
    2g & \hbox{if }3\nmid a, \\
    18g & \hbox{if }3\mid a\hbox{ and }v_3(a)\leq v_3(b), \\
    54g & \hbox{if }v_3(a)>v_3(b).
\end{cases}$$

\section{Freeness over $\mathfrak{A}_H(\mathbb{Z}[\alpha])$}\label{sec:freeness}

In this section, we study the $\mathfrak{A}_H(\mathbb{Z}[\alpha])$-freeness of $\mathbb{Z}[\alpha]$, and in particular we shall prove Theorem \ref{thm:maintheorem}. In addition, we shall specialize our findings to the case that $\mathcal{O}_L=\mathbb{Z}[\alpha]$.

Let $\beta=\beta_1+\beta_2\alpha+\beta_3\alpha^2\in\mathbb{Z}[\alpha]$, where $\beta_i\in\mathbb{Z}$ for $i\in\{1,2,3\}$. Then $$M_{\beta}(H_W,L_B)=\begin{pmatrix}
    \beta_1 & -4a^2\beta_2+6ab\beta_3 & 2\beta_1+2a\beta_3 \\
    \beta_2 & 9b\beta_2-2a^2\beta_3 & -\beta_2 \\
    \beta_3 & 6a\beta_2-9b\beta_3 & -\beta_3
\end{pmatrix}.$$

The determinant of this matrix is $$D_{\beta}(H,L)\coloneqq D_{\beta}(H_W,L_B)=2(3\beta_1+2a\beta_3)(3a\beta_2^2-9b\beta_2\beta_3+a^2\beta_3^2).$$ Recall from Proposition \ref{pro:criterionfreeness} that $\beta$ is an $\mathfrak{A}_H(\mathbb{Z}[\alpha])$-free generator of $\mathbb{Z}[\alpha]$ if and only if $|D_{\beta}(H,L)|=I_W(\mathbb{Z}[\alpha])$.

\begin{thm}\label{thm:maintheorem1} Suppose that $3\nmid a$. The following statements are equivalent:
\begin{enumerate}
    \item\label{thm:main11} $\mathbb{Z}[\alpha]$ is $\mathfrak{A}_H(\mathbb{Z}[\alpha])$-free.
    \item\label{thm:main12} There are $\beta_1,\beta_2,\beta_3\in\mathbb{Z}$ and $r,s\in\{-1,1\}$ such that $$\begin{cases}
        3\beta_1+2a\beta_3=r \\
        3\frac{a}{g}\beta_2^2-9\frac{b}{g}\beta_2\beta_3+\frac{a^2}{g}\beta_3^2=s
    \end{cases}$$
    \item\label{thm:main13} At least one of the equations $x^2+3\Delta y^2=\pm12ag$ has a solution $(x,y)\in\mathbb{Z}^2$ such that $6a\mid 9by+x$ or $9by-x$ and $3\nmid y$.
\end{enumerate}
If this is the case, a free generator is given by $\beta=\beta_1+\beta_2\alpha+\beta_3\alpha^2$, where $$\beta_1=\frac{\pm1-2ay}{3},\quad\beta_2=\frac{9by\pm x}{6a},\quad\beta_3=y.$$
\end{thm}
\begin{proof}
We know that $I_W(\mathbb{Z}[\alpha])=2g$. Then $|D_{\beta}(H,L)|=I_W(\mathbb{Z}[\alpha])$ if and only if \begin{equation}\label{eq:mainthm1}
        (3\beta_1+2a\beta_3)\Big(3\frac{a}{g}\beta_2^2-9\frac{b}{g}\beta_2\beta_3+\frac{a^2}{g}\beta_3^2\Big)=\pm1,
    \end{equation} This is equivalent to
    \begin{equation}\label{eq:mainthm1l}
        3\beta_1+2a\beta_3=r,
    \end{equation}
    \begin{equation}\label{eq:mainthm1q}
        3\frac{a}{g}\beta_2^2-9\frac{b}{g}\beta_2\beta_3+\frac{a^2}{g}\beta_3^2=s,
    \end{equation} where $r,s\in\{-1,1\}$. This proves the equivalence between \ref{thm:main11} and \ref{thm:main12}.

    Let us prove that \ref{thm:main12} and \ref{thm:main13} are equivalent. We treat \eqref{eq:mainthm1q} as a second degree equation on $\beta_2$ with parameter $\beta_3$. Then, it has discriminant \begin{equation*}
        \begin{split}
            \delta&=\Big(\frac{9b}{g}\Big)^2\beta_3^2-4\frac{3a}{g}\Big(\frac{a^2}{g}\beta_3^2-s\Big)
            \\&=\frac{81b^2}{g^2}\beta_3^2-\frac{12a^3}{g^2}\beta_3^2+\frac{12as}{g}
            \\&=\frac{81b^2-12a^3}{g^2}\beta_3^2+\frac{12as}{g}
            \\&=\frac{-3\Delta\beta_3^2+12ags}{g^2}
        \end{split}
    \end{equation*} This is a perfect square if and only if the equation $x^2+3\Delta y^2=12ags$ has some solution $(x,y)$. Suppose that this is the case. Then, the equation \eqref{eq:mainthm1q} has solutions $$\frac{9b\beta_3\pm x}{6a}.$$ At least one of these roots is an integer number if and only if $6a\mid 9b+x$ or $9b-x$. If that is the case, pick $\beta_3=y$ and $\beta_2=\frac{9b\pm x}{6a}$. Then the equation \eqref{eq:mainthm1l} on $\beta_1$ has solution $\beta_1=\frac{r-2ay}{3}$, which is an integer if and only if $3\mid 1-2ay$ or $3\mid-1-2ay$. Let us see that this is equivalent to $3\nmid y$. Indeed, it holds if and only if $3\mid(1-2ay)(-1-2ay)=4a^2y^2-1$, which is equivalent to $a^2y^2\equiv1\,(\mathrm{mod}\,3)$. Now, note that $a^2\equiv1\,(\mathrm{mod}\,3)$ because $3\nmid a$. Hence, $3\mid 1-2ay$ or $3\mid-1-2ay$ if and only if $y^2\equiv1\,(\mathrm{mod}\,3)$, that is, $3\nmid y$. Therefore, \eqref{eq:mainthm1l} and \eqref{eq:mainthm1q} are satisfied if and only if the conditions at Theorem \ref{thm:maintheorem1} \ref{thm:main13} hold, in which case it does for the above choices of $\beta_1$, $\beta_2$ and $\beta_3$.
\end{proof}

\begin{thm}\label{thm:maintheorem2} Suppose that $3\mid a$ and $v_3(a)\leq v_3(b)$. The following statements are equivalent:
\begin{enumerate}
    \item $\mathbb{Z}[\alpha]$ is $\mathfrak{A}_H(\mathbb{Z}[\alpha])$-free.
    \item The quadratic form $[\frac{a}{g},-3\frac{b}{g},\frac{a^2}{3g}]$ represents $1$ or $-1$.
    \item At least one of the equations $x^2+3\Delta y^2=\pm36ag$ has a solution $(x,y)\in\mathbb{Z}^2$ such that $6a\mid 9by+x$ or $9by-x$.
\end{enumerate}
If this is the case, a free generator is given by $\beta=\beta_1+\beta_2\alpha+\beta_3\alpha^2$, where $$\beta_1=1-\frac{2a}{3}y,\quad\beta_2=\frac{9by\pm x}{6a},\quad\beta_3=y.$$
\end{thm}
\begin{proof}
We know that $I_W(\mathbb{Z}[\alpha])=18g$. We proceed in a similar way.  We have that $\beta$ is an $\mathfrak{A}_H(\mathbb{Z}[\alpha])$-generator of $\mathbb{Z}[\alpha]$ if and only if \begin{equation}\label{eq:mainthm2}
\Big(\beta_1+\frac{2a}{3}\beta_3\Big)\Big(\frac{a}{g}\beta_2^2-3\frac{b}{g}\beta_2\beta_3+\frac{a^2}{3g}\beta_3^2\Big)=\pm1.
\end{equation} It is immediate that this equation is solvable if and only if for $s\in\{-1,1\}$, the equation $$\frac{a}{g}\beta_2^2-3\frac{b}{g}\beta_2\beta_3+\frac{a^2}{3g}\beta_3^2=s$$ on $\beta_2$ and $\beta_3$ has integer solutions. We treat it as an equation on $\beta_2$ with parameter $\beta_3$, and then it has discriminant $$\delta=\frac{-3\Delta\beta_3^2+36ags}{9g^2},$$ which is a perfect square if and only if the equation $$x^2+3\Delta y^2=36ags$$ has some solution $(x,y)\in\mathbb{Z}^2$. In that case, \eqref{eq:mainthm2} is satisfied only for $\beta_1=1-\frac{2a}{3}y$, $\beta_2=\frac{9by\pm x}{6a}$ and $\beta_3=y$, and $\beta_2\in\mathbb{Z}$ if and only if $6a\mid 9by+x$ or $9by-x$.
\end{proof}

\begin{thm}\label{thm:maintheorem3} Suppose that $v_3(a)>v_3(b)$. The following statements are equivalent:
\begin{enumerate}
    \item $\mathbb{Z}[\alpha]$ is $\mathfrak{A}_H(\mathbb{Z}[\alpha])$-free.
    \item The quadratic form $[\frac{a}{3g},-\frac{b}{g},\frac{a^2}{9g}]$ represents $1$ or $-1$.
    \item At least one of the equations $x^2+3\Delta y^2=\pm108ag$ has a solution $(x,y)\in\mathbb{Z}^2$ such that $6a\mid 9by+x$ or $9by-x$.
\end{enumerate}
If this is the case, a free generator is given by $\beta=\beta_1+\beta_2\alpha+\beta_3\alpha^2$, where $$\beta_1=1-\frac{2a}{3}y,\quad\beta_2=\frac{9by\pm x}{6a},\quad\beta_3=y.$$
\end{thm}
\begin{proof}
\begin{enumerate}
We know that $I_W(\mathbb{Z}[\alpha])=54g$. In this case, we have that $\beta$ is an $\mathfrak{A}_H(\mathbb{Z}[\alpha])$-generator of $\mathbb{Z}[\alpha]$ if and only if \begin{equation}\label{eq:mainthm3}
\Big(\beta_1+\frac{2a}{3}\beta_3\Big)\Big(\frac{a}{3g}\beta_2^2-\frac{b}{g}\beta_2\beta_3+\frac{a^2}{9g}\beta_3^2\Big)=\pm1,
\end{equation} which is equivalent to the solvability of the equation $$\frac{a}{3g}\beta_2^2-\frac{b}{g}\beta_2\beta_3+\frac{a^2}{9g}\beta_3^2=s$$ on $\beta_2$ and $\beta_3$, where $s\in\{-1,1\}$. We regard it as an equation on $\beta_2$ with parameter $\beta_3$, so it has discriminant $$\delta=\frac{-3\Delta\beta_3^2+108ags}{81g^2}.$$ This is a perfect square if and only if there are $x,y\in\mathbb{Z}$ such that $x^2+3\Delta y^2=108ags$. If so, the equation \eqref{eq:mainthm3} is only satisfied for $\beta_1=1-\frac{2a}{3}y$, $\beta_2=\frac{9by\pm x}{6a}$ and $\beta_3=y$. Moreover, $\beta_2$ is an integer if and only if $6a\mid 9by+x$ or $9by-x$.
\end{enumerate}
\end{proof}

The validity of Theorem \ref{thm:maintheorem} is established from Theorems \ref{thm:maintheorem1}, \ref{thm:maintheorem2} and \ref{thm:maintheorem3}.

\subsection{The case $\mathcal{O}_L=\mathbb{Z}[\alpha]$}

When the $\mathbb{Z}$-order $\mathbb{Z}[\alpha]$ in $L$ is the full ring of integers $\mathcal{O}_L$, the criteria above refer to the $\mathfrak{A}_H$-freeness of $\mathcal{O}_L$. In this part, we characterize when we have such an equality and specialize Theorem \ref{thm:maintheorem} to this situation.

Alaca \cite{alaca} found the explicit form of an integral basis of $L$ in terms of $\alpha$. His description is based on considering an analogous problem for each prime number $p$, namely to find a basis of $K$ that also serves as a basis for all elements of $K$ that do not admit a power of $p$ as a denominator (where $K$ is viewed as the field of fractions of its ring of integers). These are the so-called $p$-integral bases.

\begin{defi} Let $p$ be a prime number. 
\begin{enumerate}
    \item An element $x\in L$ is said to be $p$-integral if each prime ideal $\mathfrak{p}$ of $\mathcal{O}_L$ over $p$ appears with non-negative exponent at the factorization of the fracional ideal $\alpha\mathcal{O}_L$.
    \item A $p$-integral basis of $L$ is a $\mathbb{Q}$-basis $B=\{\omega_j\}_{j=1}^3$ of $L$ such that each $\omega_j$ is $p$-integral and for each $p$-integral element $x\in L$ there are unique $p$-integral elements $x_j\in\mathbb{Q}$ such that $x=\sum_{j=1}^3x_j\omega_j$.
\end{enumerate}
\end{defi}

In \cite[Theorem 1.3]{alaca}, Alaca shows that $p$-integral bases for $L$ for all prime numbers $p$ can be used to build an integral basis for $L$. In particular, we deduce the following.

\begin{coro} If $\{1,\alpha,\alpha^2\}$ is a $p$-integral basis for $L$ for all prime numbers $p$, then it is also an integral basis for $L$.
\end{coro}

Moreover, in \cite[Tables A, B and C]{alaca}, an explicit form of a $p$-integral basis for $L$ is given in terms of conditions depending on $a$, $b$. Looking at the conditions for which we obtain $p$-integral basis $\{1,\alpha,\alpha^2\}$, we can characterize when this is actually an integral basis for $L$.

\begin{coro}\label{coro:ordereqring} The field $L$ admits integral basis $\{1,\alpha,\alpha^2\}$ (that is, $\mathcal{O}_L=\mathbb{Z}[\alpha]$) if and only if for each prime $p$, exactly one of the given conditions at each item holds:
\begin{enumerate}
    \item\label{coro:ordereqring1} If $p=2$, \begin{enumerate}
        \item $b\equiv1\,(\mathrm{mod}\,2)$,
        \item $a\equiv0\,(\mathrm{mod}\,2)$ and $b\equiv2\,(\mathrm{mod}\,4)$,
        \item $a\equiv3\,(\mathrm{mod}\,4)$ and $b\equiv0\,(\mathrm{mod}\,4)$,
        \item $a\equiv1\,(\mathrm{mod}\,4)$ and $b\equiv2\,(\mathrm{mod}\,4)$.
    \end{enumerate}
    \item\label{coro:ordereqring2} If $p=3$, \begin{enumerate}
        \item $v_3(a)=0$,
        \item $v_3(a)\geq1$ and $v_3(b)=1$,
        \item $v_3(a)\geq1$, $a\not\equiv3\,(\mathrm{mod}\,9)$, $v_3(b)=0$ and $b^2\not\equiv a+1\,(\mathrm{mod}\,9)$,
        \item $a\equiv3\,(\mathrm{mod}\,9)$, $v_3(b)=0$ and $b^2\not\equiv4\,(\mathrm{mod}\,9)$.
    \end{enumerate}
    \item\label{coro:ordereqring3} If $p>3$, \begin{enumerate}
        \item $v_p(a)=0$ and $v_p(b)\geq1$,
        \item $v_p(a)\geq1$ and $v_p(b)\leq 1$,
        \item $v_p(a)=v_p(b)=0$ and $v_p(\Delta)\leq1$.
    \end{enumerate}
\end{enumerate}
\end{coro}

Combining Theorem \ref{thm:maintheorem} with Corollary \ref{coro:ordereqring}, we obtain criteria for the freeness of $\mathcal{O}_L$ as an $\mathfrak{A}_H$-module for some particular cases.

\begin{coro}\label{coro:freenesscubic} Let $L=\mathbb{Q}(\alpha)$ be a cubic number field with $\alpha$ root of a polynomial as in \eqref{eq:polynf} and let $H$ be its Hopf-Galois structure. Then $\mathcal{O}_L$ is $\mathfrak{A}_H$-free if and only if there exist integers $x,y\in\mathbb{Z}$ such that:
\begin{enumerate}
    \item $x^2+3\Delta y^2=\pm12ag$, $6a\mid 9by+x$ or $9by-x$ and $3\nmid y$, if $3\nmid a$ and Corollary \ref{coro:ordereqring} \ref{coro:ordereqring1}, \ref{coro:ordereqring3} hold.
    \item $x^2+3\Delta y^2=\pm36ag$ and $6a\mid 9by+x$ or $9by-x$, if $v_3(a)=v_3(b)=1$ and Corollary \ref{coro:ordereqring} \ref{coro:ordereqring1}, \ref{coro:ordereqring3} hold.
    \item $x^2+3\Delta y^2=\pm108ag$ and $6a\mid 9by+x$ or $9by-x$, if $v_3(a)>v_3(b)$ and Corollary \ref{coro:ordereqring} \ref{coro:ordereqring1}, \ref{coro:ordereqring2}, \ref{coro:ordereqring3} hold.
\end{enumerate}
Moreover, the sufficient conditions at each statement are equivalent to $\mathcal{O}_L=\mathbb{Z}[\alpha]$.
\end{coro}
\begin{proof}
From Corollary \ref{coro:ordereqring}, $\mathcal{O}_L=\mathbb{Z}[\alpha]$ if and only if Corollary \ref{coro:ordereqring} \ref{coro:ordereqring1}, \ref{coro:ordereqring2}, \ref{coro:ordereqring3} hold. Together with Theorem \ref{thm:maintheorem3}, this gives immediately the third statement.

If $3\nmid a$, the freeness of $\mathbb{Z}[\alpha]$ as an $\mathfrak{A}_H(\mathbb{Z}[\alpha])$-module is given by Theorem \ref{thm:maintheorem1}.  Now, this assumption is just the first condition at Corollary \ref{coro:ordereqring} \ref{coro:ordereqring2}, so it is automatically satisfied. This proves the first statement.

If $3\mid a$ and $v_3(a)\leq v_3(b)$, the freeness of $\mathbb{Z}[\alpha]$ as an $\mathfrak{A}_H(\mathbb{Z}[\alpha])$-module is given by Theorem \ref{thm:maintheorem2}. In this case,  Corollary \ref{coro:ordereqring} \ref{coro:ordereqring2} holds if and only if $v_3(a)=v_3(b)=1$, proving the second statement.
\end{proof}

\section*{Acknowledgements}

This work was supported by the grant PID2022-136944NB-I00 (Ministerio de Ciencia e Innovación), and by Czech Science Foundation, grant 24-11088O.

\printbibliography

\end{document}